\documentclass[10pt]{amsart}

\usepackage[latin1]{inputenc}
\usepackage{amsmath, amsthm, amssymb, graphicx,booktabs}
\usepackage[stable]{footmisc}
\usepackage{tikz,nicefrac}
\usetikzlibrary{calc,matrix,arrows,decorations.markings}
\usepackage[colorlinks]{hyperref}

\newtheorem{defin}{Definition}[section]
\newtheorem{theorem}[defin]{Theorem}

\newcommand{\R}{\mathbb{R}}

\newcommand{\Z}{\mathbb{Z}}

\DeclareMathOperator{\vol}{vol}

\DeclareMathOperator{\Aut}{Aut}

\newcommand{\dates}[2]{$\mathnormal{#1}$--$\mathnormal{#2}$}

\begin{document}

\title[A Breakthrough in Sphere Packing: The Search for Magic Functions]{A Breakthrough in Sphere Packing:\\ The Search for Magic Functions}

\author{David de Laat}
\address{D.~de Laat, Centrum Wiskunde \& Informatica (CWI), Science Park 123, 1098 XG Amsterdam, The Netherlands} 
\email{mail@daviddelaat.nl}

\author{Frank Vallentin} 
\address{F.~Vallentin, Mathematisches Institut, Universit\"at zu K\"oln, Weyertal 86--90, 50931 K\"oln, Germany}
\email{frank.vallentin@uni-koeln.de}

\date{8 July 2016}

\begin{abstract}
  This paper is an exposition, written for the Nieuw Archief voor
  Wiskunde,  about the two recent breakthrough results in the theory of sphere
  packings. It includes an interview with Henry Cohn, Abhinav Kumar,
  Stephen D. Miller, and Maryna Viazovska.
\end{abstract}

\maketitle

\section{Introduction}

The sphere packing problem asks for a densest packing of congruent
solid spheres in $n$-dimensional space $\R^n$. In a packing the
(solid) spheres are allowed to touch on their boundaries, but their interiors
should not intersect.

While the case of the real line, $n = 1$, is trivial, the case $n = 2$
of packing circles in the plane was first solved in 1892 by the
Norwegian mathematician Thue (\dates{1863}{1922}). He showed that the
honeycomb hexagonal lattice gives an optimal packing; see
Figure~\ref{fig:honeycomb}.

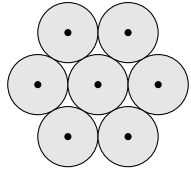
\begin{figure}[htb]
\begin{tikzpicture}[scale=2]
\foreach \x/\y in {0/0, 0.4/0, 0.8/0, 0.2/0.346, 0.6/0.346, 0.2/-0.346, 0.6/-0.346}
{
\filldraw[fill=black!10] (\x,\y) circle (0.2);
\filldraw[fill] (\x,\y) circle (0.02);
}
\end{tikzpicture}
\caption{The hexagonal lattice and the corresponding circle packing.}
\label{fig:honeycomb}
\end{figure}

The first rigorous proof is by the Hungarian mathematician Fejes
T\'oth (\dates{1915}{2005}) in 1940. He also proved that this packing
is unique (up to rotations, translations, and uniform scaling) among
periodic packings. For $n = 3$, the sphere packing problem is known as
the Kepler conjecture. It was solved by the American mathematician
Hales in 1998 following an approach by Fejes T\'oth. Hales' proof is
extremely complex, takes more than 300 pages, and makes heavy use of
computers. One of the difficulties of the sphere packing problem in
three dimensions is that there are uncountably many inequivalent
optimal packings. In 2014 a fully computer verified version of Hales'
proof was completed; it was a result of the collaborative Flyspeck
project, also directed by Hales \cite{HalesEtc}.

Recently, Maryna Viazovska, a postdoctoral researcher from Ukraine
working at the Humboldt University of Berlin, solved the
eight-dimensional case. On 14 March 2016 she announced her spectacular
result in the paper titled ``The sphere packing problem in dimension
$8$'' \cite{Viazovska2016a} on the arXiv-preprint server. Only one
week later, on 21 March 2016, Henry Cohn, Abhinav Kumar, Stephen
D.~Miller, Danylo Radchenko, and Maryna Viazovska announced a proof
for the $n=24$ case \cite{Cohn2016b}, building on Viazovska's work.

Here we want to illustrate that the optimal sphere packings in
dimensions $8$ and $24$ are very special (in
Section~\ref{sec:construction} we give constructions of the
$\mathsf{E}_8$ lattice and of the Leech lattice $\Lambda_{24}$, which
provide the optimal sphere packings in their dimensions), and we aim
to explain the main ideas of the recent breakthrough results in sphere
packing:

\begin{theorem}
  The lattice $\mathsf{E}_8$ is the densest packing in
  $\mathbb{R}^8$.  The Leech lattice $\Lambda_{24}$ is the densest packing in
  $\mathbb{R}^{24}$. Moreover, no other periodic packing achieves the
  same density in the corresponding dimension.
\end{theorem}

In Section~\ref{sec:breakthrough} we will see that the beautiful
proofs of these theorems use ideas from analytic number
theory. Viazovska found a ``magic'' function for dimension $8$, which
together with the linear programming bound of Cohn and Elkies, as
explained in Section~\ref{sec:lpbound}, gives a proof for the
optimality of the $\mathsf{E}_8$ lattice. Her method gave a hint how
to find a magic function for dimension $24$. Although the proof is
relatively easy to understand, and basically no computer assistance is
needed for its verification, computer assistance was crucial to
conjecture the existence of, and to find, these magic functions; see
Section~\ref{sec:evidence}.

\section{Optimal Lattices: Construction, Properties, Appearance}
\label{sec:construction}

In this section we introduce the two exceptional sphere packings in
dimension $8$ and $24$. The book ``Sphere Packings, Lattices, and
Groups'' of Conway and Sloane \cite{Conway1988a} is the definitive
reference on this topic; the Italian-American combinatorialist Rota
(\dates{1932}{1999}) reviewed the book saying:

\begin{quote}
  ``This is the best survey of the best work in one of the best fields
  of combinatorics, written by the best people. It will make the best
  reading by the best students interested in the best mathematics that
  is now going on.''
\end{quote}

\subsection{Lattice Packings}

How does one define a packing of unit spheres in $n$-dimensional
space? In general, such a packing is defined by the set of centers $L$
of the spheres in the packing. 

We talk about \textit{lattice packings} when $L$ forms a
\textit{lattice}. Then there are $n$ linearly independent vectors
$b_1, \ldots, b_n \in L$, called a \textit{lattice basis} of $L$, so
that $L$ is the set of integral linear combinations of
$b_1, \ldots, b_n$. For instance, the three lattices
\[
\mathbb{Z}, \quad 
\mathbb{Z} \begin{pmatrix} 2 \\ 0 \end{pmatrix} +
\mathbb{Z} \begin{pmatrix} -1 \\ \sqrt{3}\end{pmatrix}, \quad
\mathbb{Z} \begin{pmatrix} -\sqrt{2} \\ -\sqrt{2} \\ 0 \end{pmatrix} +
\mathbb{Z} \begin{pmatrix} \sqrt{2} \\ -\sqrt{2} \\ 0 \end{pmatrix} +
\mathbb{Z} \begin{pmatrix} 0 \\ \sqrt{2} \\ -\sqrt{2} \end{pmatrix}
\]
define densest sphere packings in dimensions $1$, $2$, and $3$.

We should also define what we mean when we talk about
\textit{density}. Intuitively, the density of a sphere packing is the
fraction of space covered by the spheres of the packing.  When the
sphere packing is a lattice, this intuition is easy to make precise:
The density of the sphere packing determined by $L$ is the volume of
one sphere divided by the volume of the lattice' fundamental domain. One
possible fundamental domain of $L$ is given by the parallelepiped
spanned by the lattice basis $b_1, \ldots, b_n$, so that we have
\[
\vol(\mathbb{R}^n / L) = |\det((b_1, \ldots, b_n))|.
\]
The density of $L$ is then given by
\[
\Delta(L) = \frac{\vol(B_n(r_1/2))}{\vol(\mathbb{R}^n / L)},
\]
where $r_1$ is the shortest nonzero vector length in $L$. Here,
$B_n(r_1/2)$ is the solid sphere of radius $r_1/2$, whose volume is
\[
\vol(B_n(r_1/2)) = (r_1/2)^n \frac{\pi^{n/2}}{\Gamma(n/2+1)},
\]
where $\Gamma$ is the gamma function; it satisfies the equation
$\Gamma(x+1) = x\Gamma(x)$, and two particularly useful values are
$\Gamma(1) = 1$ and $\Gamma(1/2) = \sqrt{\pi}$.

The optimal sphere packing density can be approached arbitrarily well
by the density of a periodic packing. In a \textit{periodic packing}
the set of centers is the union of a finite number $m$ of translates
of a lattice $L$
\[
\{x_1 + v : v \in L\} \cup \{x_2 + v : v \in L\} \cup \ldots \cup
\{x_m + v: v \in L\},
\]
see Figure~\ref{fig:periodic} for an example where $m = 3$.  The
density of a periodic packing is
$m \cdot \vol(B_n(r)) / \vol(\mathbb{R}^n / L)$, where $r$ is the
radius of the spheres in the packing. It is possible that in some
dimensions the optimal sphere packing is not a lattice packing; for
example, the best known sphere packing in $\R^{10}$ is a periodic
packing but it is not a lattice packing.

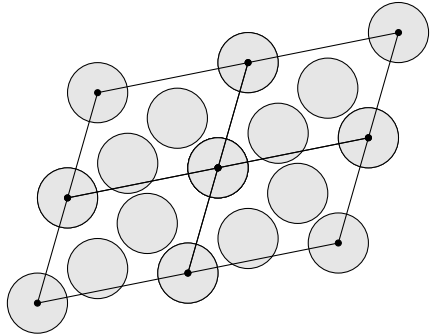
\begin{figure}[htb]

\begin{tikzpicture}[scale=2]

\foreach \x/\y in {0/0, 1/0.2, 0.2/0.7, 1.2/0.9}
{
\filldraw[xshift=\x cm,yshift=\y cm,fill=black!10] (0,0) circle (0.2);
\filldraw[xshift=\x cm,yshift=\y cm,fill=black!10] (1,0.2) circle (0.2);
\filldraw[xshift=\x cm,yshift=\y cm,fill=black!10] (0.2,0.7) circle (0.2);
\filldraw[xshift=\x cm,yshift=\y cm,fill=black!10] (1.2,0.9) circle (0.2);
\filldraw[xshift=\x cm,yshift=\y cm,fill=black!10] (0.4,0.23) circle (0.2);
\filldraw[xshift=\x cm,yshift=\y cm,fill=black!10] (0.73,0.53) circle (0.2);
}

\foreach \x/\y in {0/0, 1/0.2, 0.2/0.7, 1.2/0.9}
{
\draw[xshift=\x cm,yshift=\y cm] (0,0) -- (1,0.2);
\draw[xshift=\x cm,yshift=\y cm] (0,0) -- (0.2,0.7);
\draw[xshift=\x cm,yshift=\y cm] (0.2,0.7) -- (1.2,0.9);
\draw[xshift=\x cm,yshift=\y cm] (1,0.2) -- (1.2,0.9);
\draw[xshift=\x cm,yshift=\y cm,fill] (0,0) circle (0.02);
\draw[xshift=\x cm,yshift=\y cm,fill] (1,0.2) circle (0.02);
\draw[xshift=\x cm,yshift=\y cm,fill] (0.2,0.7) circle (0.02);
\draw[xshift=\x cm,yshift=\y cm,fill] (1.2,0.9) circle (0.02);
}

\end{tikzpicture}

\caption{A periodic packing that is not a lattice packing.}
\label{fig:periodic}
\end{figure}

Three last definitions: From a lattice $L$ we can construct its
\textit{dual lattice} by
\[
L^* = \{y \in \mathbb{R}^n : x \cdot y \in \mathbb{Z} \text{ for all }
x \in L\}.
\]
It is not difficult to see that the volume of a lattice and its dual
are reciprocal values, so that
$\vol(\R^n / L) \cdot \vol(\R^n / L^*) = 1$ holds. When a lattice
equals its dual ($L = L^*$), we call $L$ a \textit{unimodular}
lattice. In an \textit{even} lattice the square of every occurring
vector length is an even integer.

\subsection{The $\mathsf{E}_8$ Lattice}

The nicest lattices are those which are even and unimodular. However,
they only occur in higher dimensions: One can show that the first
appearance of such an even and unimodular lattice is only in dimension
$8$. It is the $\mathsf{E}_8$ lattice, which was first explicitly
constructed by the Russian mathematicians Korkine (\dates{1837}{1908)}
and Zolotareff (\dates{1847}{1878}) in 1873.

Here we give a construction of the $\mathsf{E}_8$ lattice which is
based on lifting binary error correcting codes. For this we define the
\textit{Hamming code} $\mathcal{H}_8$ via a regular three-dimensional
tetrahedron: consider the binary linear code $\mathcal{H}_8$ which is
the vector space over the finite field $\mathbb{F}_2$ (having elements
$0$ and $1$) spanned by the rows of the matrix
\[
G = (I \, | \, A) = 
\left(
\begin{array}{l|l}
1000 & 0111\\
0100 & 1011\\
0010 & 1101\\
0001 & 1110\\
\end{array}
\right)
\in \mathbb{F}_2^{4 \times 8},
\]
where $I$ is the identity matrix and where $A$ is the adjacency matrix
of the vertex-edge graph of a three-dimensional tetrahedron with
vertices $v_1, v_2, v_3, v_4$. Hence, the Hamming code is a
$4$-dimensional subspace of the vector space $\mathbb{F}_2^8$. It
consists of $2^4 = 16$ code words:
\[
\begin{array}{lllll}
0000|0000 & 1000|0111 & 1100|1100 & 0111|1000 & 1111|1111\\
                  & 0100|1011 & 1010|1010 & 1011|0100 &\\
                  & 0010|1101 & 1001|1001 & 1101|0010 &\\
                  & 0001|1110 & 0110|0110 & 1110|0001 &\\
                  &                   & 0101|0101 & & \\
                  &                   & 0011|0011 & &\\
\end{array}
\]
It is interesting to look at the occurring \textit{Hamming weights}
(the number of non-zero entries) of code words. In $\mathcal{H}_8$,
one code word has Hamming weight $0$, $14$ code words have Hamming
weight $4$, and one code word has Hamming weight $8$. Since all
occurring Hamming weights are divisible by four, and four is two times
two, the Hamming code $\mathcal{H}_8$ is called \textit{doubly even}.

Let us compute the \textit{dual code}
\[
\mathcal{H}_8^\perp = \left\{y \in \mathbb{F}_2^8 : \sum_{i=1}^8 x_i y_i =
0 \; (\text{mod } 2) \; \text{ for all } x \in \mathcal{H}_8\right\}.
\]
Squaring the matrix $A$ yields
\[
A^2_{ij} = \sum_{k = 1}^4 A_{ik} A_{kj} = |\{k : v_i \sim v_k \text{
  and } v_k \sim v_j\}| = 
\begin{cases}
3 & \text{if } i = j,\\
2 & \text{if } i \neq j.
\end{cases}
\]
Hence, $A^2 = I (\text{mod } 2))$. From this,
$GG^{\sf T} = I+ A^2 = 0 \text{ mod } 2$ follows. Hence, we have the
inclusion $\mathcal{H}_8 \subseteq \mathcal{H}_8^\perp$ and by
considering dimensions we see that $\mathcal{H}_8$ is a \textit{self-dual
code}; that is, $\mathcal{H}_8^\perp = \mathcal{H}_8$ holds.

We can define the lattice $\mathsf{E}_8$ by the following lifting
construction (which is usually called \textit{Construction A}):
\[
\mathsf{E}_8 = \left\{\frac{1}{\sqrt{2}} x : x \in \mathbb{Z}^8 ,\; x \text{ mod
    } 2 \in \mathcal{H}_8\right\}.
\]
Now it is immediate to see that $\mathsf{E}_8$ has $240$ shortest
(nonzero) vectors:
\[
\begin{array}{ll}
16 = 2^4 \text{ vectors: } & \pm \sqrt{2} e_i, \, i = 1, \ldots, 8 \\
224 = 2^4 \cdot 14 \text{ vectors: } & \frac{1}{\sqrt{2}} \sum_{i=1}^8 (\pm x_i)
                        e_i,  \; x \in \mathcal{H}_8 \text{ and }
                        \text{wt}(x) = 4,
\end{array}
\]
where $e_1, \ldots, e_8$ are the standard basis vectors of
$\mathbb{R}^8$ and where $\text{wt}(x) = |\{i : x_i \neq 0\}|$ denotes
the Hamming weight of $x$. The shortest nonzero vectors of
$\mathsf{E}_8$ have length $\sqrt{2}$. The occurring vector lengths in
$\mathsf{E}_8$ are $0, \sqrt{2}, \sqrt{4}, \sqrt{6}, \ldots$, so that
$\mathsf{E}_8$ is an even lattice.

From the lifting construction it follows that the density of
$\mathsf{E}_8$ is $|\mathcal{H}_4| = 16$ times the density of the
lattice $\sqrt{2}\Z^8$ which is spanned by
$\sqrt{2}e_1, \ldots, \sqrt{2}e_8$. Thus,
\[
\vol(\R^8 / \mathsf{E}_8) = \frac{1}{16} \cdot \vol(\R^8 / \sqrt{2}\Z^8) 
= \frac{1}{16} \cdot (\sqrt{2})^8 = 1,
\]
so that $\mathsf{E}_8$ is unimodular. One can show that $\mathsf{E}_8$
is the \textit{only} even unimodular lattice in dimension $8$. In
general, the lifting construction always yields an even unimodular
lattice when we start with a binary code which is doubly even and
self-dual.

Next to this exceptional number theoretical property, $\mathsf{E}_8$
has also exceptional geometric properties: In 1979, Odlyzko and
Sloane, and independently Levenshtein, proved that one cannot arrange
more vectors on a sphere in dimension $8$ of radius $\sqrt{2}$ so that
the distance between any two distinct vectors is also at least
$\sqrt{2}$; the $240$ vectors give the unique solution of the kissing
number problem in dimension $8$ as was shown in by Bannai and Sloane
in 1981.  Blichfeldt (\dates{1873}{1945}) showed in 1935 that
$\mathsf{E}_8$ gives the densest sphere packings among lattice
packings. For a long time it has been conjectured that $\mathsf{E}_8$
also gives the unique densest sphere packing in dimension $8$, without
imposing the (severe) restriction to lattice packings.  Now this
conjecture has been proved in the breakthrough work of Maryna
Viazovska.

\subsection{The Leech Lattice}

We turn to $24$ dimensions and to the Leech lattice. In 1965 Leech
(\dates{1926}{1992}) realized that he constructed a surprisingly dense
sphere packing in dimension $24$. For his construction, he used the
(extended binary) Golay code which is an exceptional error correcting
code found by Golay (\dates{1902}{1989}) in 1949. To define the Leech
lattice we modify the lifting construction of the $\mathsf{E}_8$
lattice. We replace the Hamming code by the Golay code and apply two
extra twists.

For defining the Golay code we replace the regular tetrahedron in the
construction of the Hamming code by the regular icosahedron and we
apply the first twist. Consider the binary code $\mathcal{G}_{24}$
spanned by the rows of the matrix
\[
G = (I \, | \, B) = 
\left(
\begin{array}{l|l}
100000000000 & 100000111111\\
010000000000 & 010110001111\\
001000000000 & 001011100111\\
000100000000 & 010101110011\\
000010000000 & 011010111001\\
000001000000 & 001101011101\\
000000100000 & 101110101100\\
000000010000 & 100111010110\\
000000001000 & 110011101010\\
000000000100 & 111001110100\\
000000000010 & 111100011010\\
000000000001 & 111111000001\\
\end{array}
\right)
\in \mathbb{F}_2^{12 \times 24},
\]
where we use $B = J - A$ (instead of $A$) and $J$ is the matrix with
only ones and where $A$ is the adjacency matrix of the vertex-edge
graph of a three-dimensional icosahedron with vertices
$v_1, \ldots, v_{12}$.  This code, the \textit{extended binary Golay
  code} $\mathcal{G}_{24}$, is a $12$-dimensional subspace in
$\mathbb{F}_2^{24}$. It contains one vector of Hamming weight $0$,
$759$ vectors of Hamming weight $8$, $2576$ vectors of Hamming weight
$16$, $759$ vectors of Hamming weight $20$, and one vector of Hamming
weight $24$; $\mathcal{G}_{24}$ is a doubly even and self-dual code.

We define the even unimodular lattice
\[
L_{24} = \left\{\frac{1}{\sqrt{2}} x : x \in \mathbb{Z}^{24}, \, x
  \text{ mod } 2 \in \mathcal{G}_{24}\right\}.
\] 
Since the the minimal non-zero Hamming weight occurring in the Golay
code is $8$, this lattice has $48$ shortest vectors
$\pm \sqrt{2} e_i$, with $i = 1, \ldots, 24$, of length $\sqrt{2}$. To
eliminate them we make the second twist, we define
\[
\begin{split}
\Lambda_{24} = & \left\{x \in L_{24} : \sqrt{2}\sum_{i=1}^{24} x_i = 0 \text{ mod 4}\right\} \\
& \qquad \cup \left\{(1, \ldots, 1) + x :
x \in L_{24}, \, \sqrt{2}\sum_{i=1}^{24} x_i = 2 \text{ mod 4}\right\}.
\end{split}
\]
\textit{This is the Leech lattice.} It is an even unimodular
lattice. In $\Lambda_{24}$ there are $196560$ shortest vectors which
have length $\sqrt{4}$. The occurring vector lengths in $\Lambda_{24}$
are $0, \sqrt{4}, \sqrt{6}, \sqrt{8}, \ldots$.

In 1969 Conway showed that the Leech lattice again has a
remarkable number theoretical property: It is the only even unimodular
lattice in dimension $24$ which does not have vectors of length
$\sqrt{2}$. He used this result to determine the automorphism group
(the group of orthogonal transformation which leave $\Lambda_{24}$
invariant) of the Leech lattice and it turned out the number of
automorphisms equals
\[
|\text{Aut}(\Lambda_{24})| = 2^{22} \cdot 3^9 \cdot 5^4 \cdot 7^2
\cdot 11 \cdot 13 \cdot 23 = 8315553613086720000,
\]
and that this group contained three new sporadic simple groups
$\mathrm{Co}_1, \mathrm{Co}_2. \mathrm{Co}_3$. The classification
theorem of finite simple groups, which was announced in 1980, says
that there are only $26$ finite simple sporadic groups. They are
sporadic in the sense that they are not contained in the infinite
families of cyclic groups of prime order, alternating groups and
groups of Lie type.

Similar to the eight-dimensional case, by results of Odlyzko, Sloane,
Levenshtein, Bannai and Sloane, the $196560$ shortest vectors of
$\Lambda_{24}$ give the unique solution of the kissing number problem
in dimension $24$. In 2004 Cohn and Kumar proved the optimality of the
sphere packing of the Leech lattice among lattice packings by a
computer assisted proof, see \cite{Cohn2009a} and
Section~\ref{sec:evidence}. However, despite all the similarities of
$\mathsf{E}_8$ and $\Lambda_{24}$, there is a puzzling difference
between $\mathsf{E}_8$ and $\Lambda_{24}$ when it comes to sphere
coverings: Sch\"urmann and Vallentin \cite{Schuermann} showed in 2006
that $\Lambda_{24}$ provides at least a locally thinnest sphere
covering in the space of $24$-dimensional lattices, whereas
Dutour-Sikiri\'c, Sch\"urmann, and Vallentin \cite{Dutour} showed in
2012 that one can improve the sphere covering of the $\mathsf{E}_8$
lattice when picking a generic direction in the space of
eight-dimensional lattices.

\subsection{Theta Series and Modular Forms}

As already indicated, the class of even unimodular lattices is
restrictive, at least in small dimensions.  One can show that they
only exist in dimensions that are divisible by $8$, furthermore for
every such dimension $n$ there are only finitely many even unimodular
lattices, this number is denoted by $h_n$. In the following table we
summarize the known values of $h_n$:
\[
\begin{array}{l|l}
h_8 = 1 & \text{Mordell, 1938}\\
h_{16} = 2 & \text{Witt, 1941}\\
h_{24} = 24 & \text{Niemeier, 1973}\\
 h_{32} \geq 1162109024 & \text{King, 2003}
\end{array}
\]

A major tool for studying even unimodular lattice are their theta
series (first studied by Jacobi (\dates{1804}{1851})): The
\textit{theta series} of a lattice $L$ is
\[
\vartheta_L = \sum_{r = 0}^\infty n_L(r) q^r \; \text{with} \; n_L(r)
= \{x \in L : x \cdot x = 2r\},
\]
the generating function of the number of lattice vectors of length
$\sqrt{2r}$. In order to work with them analytically we set
$q = e^{2\pi i z}$ where $z$ lies in the upper half
$\mathbb{H} = \{z \in \mathbb{C} : \text{Im } z > 0\}$ of the complex
plane, so that $\vartheta_L$ is a function of $z$. The theta function
is periodic mod $\mathbb{Z}$: we have
$\vartheta_L(z) = \vartheta_L(z + 1)$.  The \textit{Poisson summation
  formula} states
\[
\sum_{x \in L} f(x + v) = \frac{1}{\vol(\mathbb{R}^n / L)} \sum_{y \in
  L^*} e^{2\pi i x \cdot y} \widehat{f}(y), \text{ with } v \in \R^n,
\]
where
\[
\widehat{f}(y) = \int_{\mathbb{R}^n} f(x) e^{-2\pi i y \cdot x} dx
\]
is the $n$-dimensional \textit{Fourier transform}. Using the Poisson summation
formula one can show that $\vartheta_L$ satisfies the transformation
law
\[
\vartheta_L(-1/z) = (z/i)^{n/2} \frac{1}{\vol(\R^n/L)} \vartheta_{L^*}(z),
\]
which in particular shows that $\vartheta_L$ is a modular form of
weight $n/2$. From this it is not difficult to derive that an even unimodular
lattice can only exist when $n$ is a multiple of $8$.

\textit{What is a modular form?} The group
\[
\mathrm{SL}_2(\mathbb{Z}) = \left\{\begin{pmatrix} a & b\\c &
  d\end{pmatrix} : a,b,c,d \in \mathbb{Z}, ad-bc = 1\right\},
\]
which is generated by the matrices
\[
S = \begin{pmatrix} 0 & -1\\ 1 & 0 \end{pmatrix}
\; \text{ and } \;
T = \begin{pmatrix} 1 & 1\\ 0 & 1 \end{pmatrix}
\]
acts on upper half plane $\mathbb{H}$ by fractional linear
transformations
\[
\begin{pmatrix} a & b\\c & d\end{pmatrix} z = \frac{az + b}{cz + d}.
\]
The action of the generator $S$ corresponds to the involution
$z \mapsto -1/z$ and the action of the generator $T$ corresponds to
the translation $z \mapsto z + 1$.

\begin{figure}[htb]

\begin{tikzpicture}[scale=2.3]

\fill[color=black!15] (0,1) arc (90:120:1) -- (-0.5, 1.8) -- (0.5,1.8) -- (0.5,0.866) arc (60:90:1) -- cycle; 

\draw (-1.7,0) -- (1.7,0);
\draw (-1.5,0.866) -- (-1.5,1.8);
\draw (-1,1) arc(90:120:1);
\draw (-0.5,0.866) -- (-0.5,1.8);
\draw (0,1) arc(90:120:1);
\draw (0.5,0.866) -- (0.5,1.8);
\draw (1,1) arc(90:120:1);
\draw (1.5,0.866) -- (1.5,1.8);

\draw (-1.5,0.866) arc(60:0:1);
\draw (-1,1) arc(90:60:1);
\draw (-0.5,0.866) arc(120:180:1);
\draw (0,1) arc(90:60:1);
\draw (0.5,0.866) arc(120:180:1);
\draw (1,1) arc(90:60:1);
\draw (1.5,0.866) arc(120:180:1);

\draw (-0.5,0.577) -- (-0.5,0.287);
\draw (-0.5,0.287) arc(60:180:1/3);
\draw (-0.5,0.577) -- (-0.5,0.866);
\draw (-0.5,0.866) arc(60:0:1);
\draw (-0.5,0.287) arc(120:0:1/3);

\draw (0.5,0.577) -- (0.5,0.287);
\draw (0.5,0.287) arc(60:180:1/3);
\draw (0.5,0.577) -- (0.5,0.866);
\draw (0.5,0.866) arc(60:0:1);
\draw (0.5,0.287) arc(120:0:1/3);

\draw (-1,1.4) node {$T^{-1}$};
\draw (0,1.4) node {$I$};
\draw (1,1.4) node {$T$};
\draw (-1,0.75) node {$(ST)^2$};
\draw (0,0.75) node {$S$};
\draw (1,0.75) node {$TS$};
\draw (-0.70,0.46) node {$STS$};
\draw (-0.33,0.46) node {$ST$};
\draw (0.31,0.46) node {$ST^{-1}$};
\draw (0.68,0.46) node {$TST$};

\draw[fill] (0,1) circle (0.015);
\draw (0,0.99) node[black,above] {$i$};

\draw (-1.5,0.03) -- (-1.5,-0.03) node[below] {$-3/2$};
\draw (-1,0.03) -- (-1,-0.03) node[below] {$-1$};
\draw (-0.5,0.03) -- (-0.5,-0.03) node[below] {$-1/2$};
\draw (0,0.03) -- (0,-0.03) node[below] {$0$};
\draw (0.5,0.03) -- (0.5,-0.03) node[below] {$1/2$};
\draw (1,0.03) -- (1,-0.03) node[below] {$1$};
\draw (1.5,0.03) -- (1.5,-0.03) node[below] {$3/2$};

\end{tikzpicture}

\caption{A fundamental domain of the action of $\mathrm{SL}_2(\mathbb{Z})$ on the upper half plane $\mathbb{H}$.}
\end{figure}
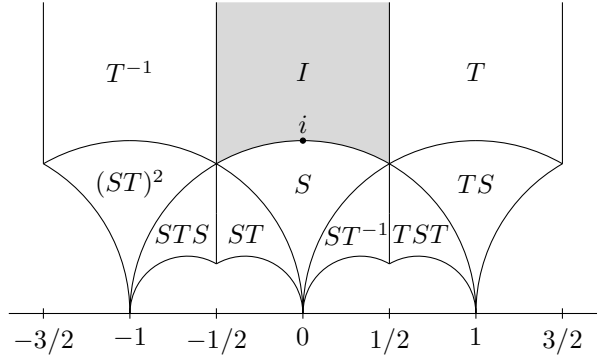

A \textit{modular form of weight $k$} is a holomorphic function $f \colon
\mathbb{H} \to \mathbb{C}$ that satisfies the transformation law
\[
f\left(\frac{az + b}{cz + d}\right) = (cz + d)^k f(z) \; \text{ for all }
\; \begin{pmatrix} a & b\\c & d\end{pmatrix} \in
\mathrm{SL}_2(\mathbb{Z}), \; z \in \mathbb{H},
\]
and which has a power series expansion in $q = e^{2\pi i z}$.

Next to theta series, \textit{Eisenstein series}, due to Eisenstein
(\dates{1823}{1852}), form an important class of modular forms. For an 
integer $k \geq 3$ define
\begin{equation}
\label{eq:Eisenstein}
E_k(z) = \frac{1}{2\zeta(k)} \sum_{(c,d) \in \mathbb{Z}^2 \setminus
  \{0\}} \frac{1}{(cz+d)^k},
\end{equation}
where $\zeta$ is the Riemann zeta function $\zeta(s) =
\sum_{n=1}^\infty \frac{1}{n^s}$. For even integers $k \geq 3$, the Eisenstein series $E_k$ is a modular form of weight $k$.

Curiously, a theorem of Siegel (\dates{1896}{1981}) gives a relation
between the theta series of even unimodular lattices and Eisenstein
series. Let $L_1, \ldots, L_{h_n}$ be the set of even unimodular
lattices in dimension $n$. Define
\[
M(n) = \frac{1}{|\Aut(L_1)|} + \cdots + \frac{1}{|\Aut(L_{h_n})|}.
\]
Then 
\[
E_{n/2}(z) = \frac{1}{M(n)} \sum_{j=1}^{h_n} \frac{1}{|\Aut(L_j)|} \vartheta_{L_j}(z)
\]

Another striking fact is that one can show that the modular forms form
an algebra which is isomorphic to the polynomial algebra
$\mathbb{C}[E_4,E_6]$.

When $k$ is even, the Eisenstein series has the Fourier expansion
\[
E_k(z) = 1 + \frac{2}{\zeta(1-k)} \sum_{n=1}^\infty \sigma_{k-1}(n)
e^{2\pi i n z},
\]
where $\sigma_{k-1}(n) = \sum_{d | n} d^{k-1}$ is the divisor
function. We can express the theta series of the $\mathsf{E}_8$
lattice and the Leech lattice $\Lambda_{24}$ by $E_4$ and $E_6$. For
the $\mathsf{E}_8$ lattice we have
\begin{align*}
\vartheta_{\mathsf{E}_8}(z) &= E_4(z) = 1 + 240 \sum_{r=1}^\infty \sigma_3(r) q^r \\
&= 1 + 240 \cdot 1 \cdot q + 240 \cdot 9 \cdot q^2 + 240 \cdot 28 \cdot q^3 + \cdots\\
&= 1 + 240 \cdot q + 2160 \cdot q^2 + 6720 \cdot q^3 + \cdots 
\end{align*}
For the Leech lattice we have
\[
\vartheta_{\Lambda_{24}}(z) = E^3_4(z) - 720 \Delta(z) = 1 + 196560
\cdot q^2 + 16773120 \cdot q^3 + \cdots 
\]
where $\Delta$ is defined as
\[
\Delta(z) = \frac{E_4(z)^3 - E_6(z)^2}{1728}.
\]

When $k = 2$, we can still write down the series as
in~\eqref{eq:Eisenstein}, but then we loose some pleasant
properties. For example the series no longer converges absolutely, so
the ordering of summating matters. Then
\[
E_2(z) = \frac{1}{2\zeta(2)} \sum_{c \in \mathbb{Z}} \sum_{d \in
  \mathbb{Z}}  \frac{1}{(c+dz)^2} = 
1 - 24 \sum_{n=1}^\infty \sigma_1(n) e^{2\pi i n z},
\]
where we of course omit the pair $(c,d) = (0,0)$ in the first sum.
This \textit{forbidden Eisenstein series} is not a modular form;
instead it satisfies the following transformation law
\[
E_2(-1/z) = z^2 E_2(z) - \frac{6iz}{\pi}.
\]
It is a \textit{quasi-modular form}.

\section{The Linear Programming Bound of Cohn and Elkies}
\label{sec:lpbound}

We can use optimization techniques, in particular linear and
semidefinite programming, to obtain upper bounds on the optimal sphere
packing density.

Let us recall some facts about linear programming. In a \emph{linear
  program} we want to maximize a linear functional over a
polyhedron. For example, we maximize the functional
$a \mapsto c \cdot a$ over all (entrywise) nonnegative vectors
$a \in \R^d$ satisfying the linear system $Aa = b$, or we maximize the
functional over all (nonnegative) vectors $a$ satisfying the
inequality $Aa \leq b$. Linear programs can be solved efficiently in
practice by a simplex algorithm (which runs over the vertices of the
polyhedron) or by Karmakar's interior point method, where the latter
runs in polynomial time.

The following theorem, the linear programming bound of Cohn and Elkies
from 2003, can be used to obtain upper bounds on the sphere packing
density. In the statement of this theorem we restrict to Schwartz
functions because the proof, which we give here because it is simple
and insightful, uses the Poisson summation formula. A function
$f \colon \R^n \to \R$ is a \emph{Schwartz function} if all its
partial derivatives exist and tend to zero faster than any inverse
power of $x$. There are alternative proofs that do not use Poisson
summation and for which the Schwartz condition can be weakened.

\begin{theorem} \label{thm:cohn-elkies} If $f \colon \R^n \to \R$ is a
  Schwartz function and $r_1$ is a positive scalar with
  $\smash{\widehat f}(0) = 1$, $\smash{\widehat f}(u) \geq 0$ for all
  $u$, and $f(x) \leq 0$ for $\|x\| \geq r_1$, then the density of a
  sphere packing in $\R^n$ is at most
  $f(0) \cdot \mathrm{vol}(B_n(r_1/2))$.
\end{theorem}

\begin{proof}
Let $P$ be a periodic packing of balls of radius $r_1/2$. This means there is a lattice $L$ and points $x_1,\ldots, x_m$ in $\R^n$ such that
\[
P = \bigcup_{v \in L} \bigcup_{i=1}^m (v + x_i + B_n(r_1/2)).
\]
The density of $P$ is $m \cdot \mathrm{vol}(B_n(r_1/2)) / \vol(\R^n / L)$. By Poisson summation we have
\[
\sum_{v \in L} \sum_{i,j=1}^m f(v + x_j - x_i) = \frac{1}{\vol(\R^n / L)} \sum_{u \in L^*} \widehat f(u) \sum_{i,j=1}^m e^{2\pi i u \cdot (x_j -x_i)},
\]
and because $\widehat f(0) = 1$ and $\widehat f(u) \geq 0$ for all $u$, this is at least $m^2 / \vol(\R^n / L)$.
On the other hand, by the condition $f(x) \leq 0$ for $\|x\| \geq r_1$, we have
\[
\sum_{v \in L} \sum_{i,j=1}^m f(v + x_j - x_i) \leq m f(0).
\]
Hence, the density of $P$ is at most $f(0) \cdot \mathrm{vol}(B_n(r_1/2))$. The density of any packing can be approximated arbitrarily well by the density of a periodic packing, so this completes the proof.
\end{proof}

We can additionally require either $f(0)= 1$ or $r_1 = 1$, without
weakening the theorem. Moreover, we can restrict to \textit{radial
  function}s, for if a function $f$ satisfies the conditions of the
theorem for some $r_1$, then the function
\[
x \mapsto \smash{\int_{S^{n-1}}} f(\|x\| \xi) \, d\omega(\xi),
\]
where $\omega$ is the normalized invariant measure on $S^{n-1}$, also
satisfies these conditions. For $f \colon \R^n \to \R$ radial, we
(ab)use the notation $f(r)$ for the common value of $f$ on the vectors
of length $r$. The (inverse) Fourier transform maps radial functions
to radial functions. Moreover, the \text{Gaussian}
$x \mapsto \smash{e^{-\pi \|x\|^2}}$ is a fixed point of the Fourier
transform, and, more generally, the sets
\[
P_d = \Big\{x \mapsto p(\|x\|^2) e^{-\pi \|x\|^2} : p \text{ is a polynomial of degree at most $d$} \Big\},
\]
are invariant under the Fourier transform. A
\textit{computer-assisted} approach to find good functions for the
above theorem is to restrict to functions from $P_d$ for some fixed
value of $d$. Any function from this set that satisfies $f(0)=1$ can
be written as
\begin{equation}
\label{eq:fa}
f_a(x) = \Big(1 + \sum_{k=1}^d a_kk! \pi^{-k}L_k^{n/2-1}(\pi  \|x\|^2)\Big) e^{-\pi \|x\|^2}
\end{equation}
for some $a \in \R^d$, where $L_k^{n/2-1}$ is the Laguerre polynomial
of degree $k$ with parameter $n/2-1$ (\text{Laguerre polynomials} are
a family of orthogonal polynomials). We choose this form for $f_a$, so
that its Fourier transform is
\begin{equation}
\label{eq:hatfa}
\widehat f_a(u) = \Big(1 + \sum_{k=1}^d a_k \|u\|^{2k}\Big) e^{-\pi \|u\|^2},
\end{equation}
which means that $a \geq 0$ immediately implies
$\widehat f_a(u) \geq 0$ for all $u$. Setting $r_1=1$ in the above
theorem, we see that the optimal sphere packing density is upper
bounded by the maximum of the linear functional $a \mapsto f_a(0)$
over all nonnegative vectors $a \in \R^d$ for which the linear
inequalities $f_a(r) \leq 0$ for $r > 1$ are satisfied. For every
fixed value of $d$ this gives a semi-infinite linear program, which is
a linear program with finitely many variables and infinitely many
linear constraints.

One approach to solving these semi-infinite programs is to select a
finite sample $S \subseteq [1, \infty)$ and only enforce the
constraints $f_a(r) \leq 0$ for $r \in S$. For each $S$ this yields a
linear program whose optimal solution $a$ can be computed using a
linear programming solver. Then we verify that $f_a(r) \leq 0$ for all
$r > 1$, or if this is not (almost) true, we run the problem again
with a different (typically bigger) sample $S$. This approach works
well in practice for the \emph{spherical code problem}, which is a
compact analogue of the sphere packing problem. Here, given a scalar
$t \in (-1,1)$, we seek a largest subset of the sphere
$S^{n-1} \subseteq \R^n$ such that the inner product between any two
distinct points is at least $t$. A spherical code corresponds to a
spherical cap packing where we center spherical caps of angle
$\arccos(t/2)$ about the points in the code. The Cohn-Elkies bound can
be seen as a noncompact analogue of a similar bound for the spherical
code problem known as the \textit{Delsarte linear programming
  bound}. However, because of noncompactness this sampling approach
does not work well for the sphere packing problem.

Another approach, based on semidefinite programming, does work well
for noncompact problems such as the sphere packing problem. In
\cite{LaatOliveiraVallentin2014} this approach is used to compute
upper bounds for packings of spheres and spherical caps of several
radii. \emph{Semidefinite programming} is a powerful generalization of
linear programming, where we maximize a linear functional over a
spectrahedron instead of a polyhedron. That is, we maximize a
functional $X \mapsto \langle X, C \rangle$ over all positive
semidefinite $n \times n$ matrices $X$ that satisfy the linear
constraints $\langle X, A_i \rangle = b_i$ for $i = 1, \ldots,
m$.
Here $\langle A, B \rangle = \mathrm{trace}(B^{\sf T} A)$ denotes the
trace inner product. As for linear programs, semidefinite programs can
be solved efficiently by using interior point methods. The usefulness
of semidefinite programming in solving the above semi-infinite linear
programs stems from the following two observations: Firstly, P\'olya
and Szeg\"o showed that a polynomial is nonnegative on the interval
$[1, \infty)$ if and only if it can be written as
$s_1(r) + (r-1) s_2(r)$, where $s_1$ and $s_2$ are sum of squares
polynomials. Secondly, a sum of squares polynomial of degree $2d$ can
be written as $b(r)^{\sf T} Q b(r)$, where
$b(r) = (1, r, \ldots, r^d)$, and where $Q$ is a positive semidefinite
matrix. (To see that a polynomial of this form is a sum of squares
polynomial one can use a Cholesky factorization $Q = R^{\sf T} R$.)
Using these observations we can introduce two positive semidefinite
matrix variables $Q_1$ and $Q_2$, and replace the infinite set of
linear constraints $f_a(r) \leq 0$ for $r > 1$, by a set of $2d+2$
linear constraints that enforce the identity
\[
f_a(r) = (1-r) b(r)^{\sf T} Q_2 b(r) - b(r)^{\sf T} Q_1 b(r).
\]
In this way we obtain a semidefinite program, which can be solved with
a semidefinite programming solver, and whose optimal value upper
bounds the sphere packing density. For each $d$ this finds the optimal
function $f_a$

\section{Producing Numerical Evidence}
\label{sec:evidence}

Now we want to understand what must happen when the Cohn-Elkies bound
can be used to \textit{prove the optimality} of the sphere packing
given by a unimodular lattice $L$.

Then there exists a \textit{magic function} $f \colon \R^n \to \R$ and
a scalar $r_1$ that satisfy the conditions of
Theorem~\ref{thm:cohn-elkies}, such that the density of the sphere
packing equals $f(0) \cdot \mathrm{vol}(B_n(r_1/2))$. As observed
before, we may assume $f(0) = 1$, so that $r_1$ is the shortest
nonzero vector length in $L$. Under these assumptions on $f$ we can
derive extra properties that the function $f$ and its Fourier
transform must satisfy. Since $f(x) \leq 0$ and
$\smash{\widehat f}(x) \geq 0$ for all $\|x\| \geq r_1$, and
$f(0) = \smash{\widehat f}(0) = 1$, we have \textit{equality}
\[
1 \leq \sum_{x \in L} \widehat f(x) = \sum_{x \in L} f(x) \leq 1.
\]
This says that we have to have $f(x) = \smash{\widehat f}(x) = 0$ for
all $x \in L \setminus \{0\}$. In fact, we can apply this argument to
any rotation of $L$, so that $f(x) = \smash{\widehat f}(x) = 0$ for
all $x$ where $\|x\|$ is a nonzero vector length in $L$. As noted
before, we may take $f$ to be radial, and then we have (again abusing
notation)
\[
f(r) = \widehat f(r) = 0  \; \text{for all nonzero vector lengths $r$ in $L$.}
\]
This also tells us something about the \textit{orders of the roots}.
We have $f(0) = 1$ and $f(r) \leq 0$ for $r \in [r_1,\infty)$, so the
roots at the vector lengths that are strictly larger than $r_1$ must
have even order. We have $\smash{\widehat f}(0) = 1$ and
$\smash{\widehat f}$ is nonnegative on $[0,\infty)$, so the roots at
the nonzero vector lengths must have even order.

If $f$ does not have additional roots, then in \cite{CohnElkies2003}
it is shown that there is no other periodic packing achieving the same
density as $L$. To apply this it is important that $\mathsf E_8$ is
the only even unimodular lattice in $\R^8$ and $\Lambda_{24}$ is the
only even unimodular lattice in $\R^{24}$ that does not contain
vectors of length $\sqrt{2}$.

\subsection{The Cohn-Elkies paper}

In \cite{CohnElkies2003} Cohn and Elkies used this insight about the
potential locations of the roots and double roots to derive a
numerical scheme to find functions that are \textit{close} to magic
functions. They parametrized the function $f_a$ as in
\eqref{eq:fa}. Then they required that $f_a$ and $\widehat f_a$ have
as many roots and double roots as possible, depending on the degree
$d$. Afterwards they applied Newton's method to perturb the roots and
double roots in order to optimize the value of the bound. In dimension
$8$ and $24$ they obtained bounds which were too high only by factors
of $1.000001$ and $1.0007071$. This provided the first strong evidence
that magic functions exist for these two dimensions. Since the magic
functions $f$ have to have infinitely many roots, the degree $d$ has
to go to infinity. Could this method, in the limit, actually give the
exact sphere packing upper bounds?

\subsection{The Cohn-Kumar paper}

The next step was taken by Cohn and Kumar in \cite{Cohn2009a}. They
improved the numerical scheme and by using degree $d = 803$ (with
$3000$-digit coefficients) they showed that in dimension $24$ there is
no sphere packing which is $1 + 1.65 \cdot 10^{-30}$ times denser than
the Leech lattice. The actual aim of their paper was to show that the
Leech lattice is the unique densest lattice in its dimension. For this
they used the numerical data together with the known fact that the
Leech lattice is a strict local optimum. The proof of Cohn and Kumar
is a beautiful example of the symbiotic relationship between human and
machine reasoning in mathematics.

\subsection{The Cohn-Miller paper}

For a long time Cohn and Miller were fascinated by the properties of
these conjecturally existing magic functions. In their paper
\cite{Cohn2016a}, submitted March 15, 2016 to the arXiv-preprint
server, they gave a ``construction'' of the magic functions using
determinants of Laguerre polynomials. However, they could not prove
that this construction indeed worked. With the use of high precision
numerics they experimented with their construction. This resulted in
improved bounds, optimality of $\Lambda_{24}$ within a factor of
$1 + 10^{-51}$. Even more importantly, they detected some unexpected
rationalities: For instance using their numerical data they
conjectured that for $n = 8$, the magic function $f$ and
$\smash{\widehat f}$ have quadratic Taylor coefficients $-27/10$ and
$-3/2$, respectively. For $n = 24$, the corresponding coefficients
should be $-14347/5460$ and $-205/156$.

\section{Viazovska's Breakthrough}
\label{sec:breakthrough}

Viazovska made the spectacular discovery that a magic function indeed
exists for dimension $n = 8$.  Building on this, Cohn, Kumar, Miller,
Radchenko, and Viazovska found a magic function for $n =
24$.
Viazovska's construction is based on a couple of new ideas, which we
want to explain briefly.

Each radial Schwartz function $f \colon \R^n \to \R$ can be written as
a linear combination of radial eigenfunctions of the Fourier transform
in $\R^n$ with eigenvalues $+1$ and $-1$.  Viazovska wrote the magic
function as a linear combination $f = \alpha f_+ + \beta f_-$, where
$f_+$ is a radial eigenfunction of the Fourier transform with
eigenvalue $+1$ and $f_-$ is a radial eigenfunction with eigenvalue
$-1$.  The coefficients $\alpha$ and $\beta$ are determined later on.

She makes the Ansatz that for $r > r_1$, we can write these functions
$f_+$ and $f_-$ as a squared sine function times the Laplace transform
of a (quasi)-modular form. That is, she proposes that
\[
f_+(r) = -4 \sin(\pi r^2/2)^2 \int_0^{i \infty} \psi_+ (-1/z) z^{n/2-2} e^{\pi i r^2 z} \, dz
\]
and
\[
f_-(r) = -4 \sin(\pi r^2/2)^2 \int_0^{i\infty} \psi_-(z) e^{\pi i r^2 z} \, dz,
\]
where $\psi_+$ is a quasi-modular form and $\psi_-$ is a modular
form.

The $\sin(\pi r^2/2)^2$ factor insures (assuming the above
integrals do not have cusps) that the resulting function $f$ (as well
as its Fourier transform) have double roots at all but the first
occurring vector lengths.

Viazovska noticed that the function $f_-$ is an eigenfunction of the
Fourier transform having eigenvalue $-1$ when the following modularity
relation holds:
\[
\psi_-\left(\frac{az+b}{cz+d}\right) = (cz+d)^{2-n/2} \psi_-(z)
\text{ for all } \begin{pmatrix} a & b \\ c & d\end{pmatrix} 
\in \mathrm{SL}_2(\mathbb{Z}), a,d \text{ odd, } b,c \text{ even.}
\]
For the explicit definition of $f_-$ we need the theta functions
\[
\Theta_{01} = \sum_{n \in \Z} (-1)^n e^{\pi i n^2 z} \quad \text{and} \quad
\Theta_{10} = \sum_{n \in \Z} e^{\pi i (n + 1/2)^2 z}.
\]
Similarly, but technically much more involved, the function $f_+$ is an
eigenfunction for the eigenvalue $+1$ when $\psi_+$ is a quasi-modular
form. The forbidden Eisenstein series $E_2$ becomes important here.

Now it needs to be shown that $f_+$ and $f_-$ have analytic
extensions, that there exists a linear combination
$f = \alpha f_+ + \beta f_-$ so that $\smash{\widehat f}(0) = f(0) = 1$ holds, and
that the sign conditions $f(r) \leq 0$ for $r \geq r_1$ and
$\smash{\widehat f}(u) \geq 0$ for all $u \geq 0$ are fulfilled.  For this, and
more, we refer to the beautiful original papers. Here we want to end
with taking a look at the magic functions: See Table~\ref{tab:magic}.

\begin{table}[htb]
\centering
{\scriptsize
\begin{tabular}{ll}
\toprule
$\mathsf E_8$ & $\Lambda_{24}$\\
\midrule
$\psi_+ = \dfrac{(E_2E_4 - E_6)^2}{\Delta}$ & $\psi_+ = \dfrac{25 E_4^4 - 49 E_6^2E_4 + 48 E_6E_4^2E_2 + 25E_6^2E_2^2-49E_4^3E_2^2}{\Delta^2}$\\[1.3em]
$\psi_- = \dfrac{5\Theta_{01}^{12} \Theta_{10}^8 + 5 \Theta_{01}^{16}\Theta^4_{10} + 2 \Theta_{01}^{20}}{\Delta}$ & $\psi_- = \dfrac{7\Theta_{01}^{20}\Theta_{10}^8 + 7 \Theta_{01}^{24}\Theta_{10}^4 + 2 \Theta_{01}^{28}}{\Delta^2}$\\[1.2em]
$f(x) = \dfrac{\pi i}{8640} f_+(x) + \dfrac{i}{240\pi} f_-(x)$ & $f(x) = -\dfrac{\pi i}{113218560} f_+(x) - \dfrac{i}{262080\pi} f_-(x)$\\
\bottomrule
\end{tabular}
}
\smallskip
\smallskip
\caption{The magic functions.}
\label{tab:magic}
\end{table}

\begin{appendix}
\section{Interview with Henry Cohn, Abhinav Kumar, Stephen D. Miller,
  and Maryna Viazovska\footnote{We conducted the interview using the online communication platform google hangouts between 20 May, 2016 and 14 June, 2016.}}

\section*{Computer Assistance}

\begin{description}

\item[\sc Frank] 

Dear Henry, Abhinav, Steve, and Maryna, First of all let me
congratulate you to your breakthrough papers. This issue of the
``Nieuw Archief voor Wiskunde'' is a special issue focussing on
computer-assisted mathematics. In it we have two articles about sphere
packings: One about the formal proof of the Kepler conjecture and one
about your recent breakthrough on sphere packings in dimensions $8$
and $24$. At the moment a proof of the Kepler conjecture without
computer-assistance is not in sight, but your proofs in dimension $8$
and $24$ require almost no computers. How were computers helpful to
you when finding the proofs?

\item[\sc Abhinav]

The Cohn-Elkies paper and later the Cohn-Kumar and the Cohn-Miller
papers were certainly useful as indicators that the solution was out
there waiting for the right functions. In our new
Cohn-Kumar-Miller-Radchenko-Viazovska paper at least, the numerical
data was quite useful because once we had figured out the right finite
dimensional space of modular or quasi-modular forms, the numerics
helped us pin down the exact form (up to scaling) of the function. In
particular, we matched values and derivatives of $f$ and $\widehat{f}$
at lattice vector lengths to cut down the space by imposing linear
conditions. Steve has a \texttt{Mathematica} code to do some of this
but we also used \texttt{PARI/GP} and occasionally \texttt{Maple}.

A couple more things---in both the proofs the final inequalities
needed for the functions and Fourier transforms are done with computer
assistance. There might be more elegant hand proofs, but so far we
haven't found them. And we did quite a lot of messing around with
$q$-expansions etc., which would have been very painful outside of a
computer algebra system

\item[\sc Steve]

Though I think there will eventually be a slick proof that can be
written by hand, computers were completely essential in this story.
To me the best example is the appearance of rational numbers that
Henry Cohn and I discovered.

I can only speak for myself, but I was completely fascinated by the
(then proposed) existence of these ``magic'' functions which
completely solve sphere packing in special dimensions.  To satisfy
this curiosity, Henry and I began computing their features to see if
we could learn more about them.  We found---using some serendipity
with the online ``inverse symbolic calculator'' website---that their
quadratic Taylor coefficients were rational, and furthermore related
to Bernoulli numbers.  This was a strong hint that modular forms were
connected, though we never understood why until we saw Maryna's paper.
We found other rationalities (such as the derivatives at certain
points) using some theoretical motivation, combined with good
numerical approximations.  It was a type of ``moonshine'', with a
fascinating sequence of numbers and an amazing structure we could not
otherwise access.

Once Maryna's paper appeared, it took just a few days to combine her
insight with our previous numerics in an exact way.  It's important to
stress that at this point we could derive the magic function for $24$
dimensions without using floating point calculations, since we had
already extrapolated exact expressions from previous numerics.  Maybe
we will later understand a way to derive the $24$-dimensional
functions without such information, but at the time it was highly
convenient to leverage them.

\item[\sc Frank]

Maryna, did you also use computer assistance when you found the
function for dimension $8$? In your paper you mention in passing that
one compute the first hundred terms of Fourier expansions of modular
forms in a few second using \texttt{PARI/GP} or \texttt{Mathematica}.

\item[\sc Maryna]

Numerical evidence was crucial to believe in the existence and
uniqueness of ``magic'' functions. I used computer calculations to
verify that approximations to the magic function computed from linear
equations (similar to the equations considered in Cohn-Miller paper)
converge to the function computed as an integral transform of a
modular form. I used \texttt{Mathematica} and \texttt{PARI/GP} for
this purpose.

\item[\sc Frank]

Checking the positivity conditions is the only part of the proof which
depends on the use of computers. How do you make sure that your
computer proof is indeed mathematically rigorous?

\item[\sc Henry]

In her $8$-dimensional proof, Maryna used interval arithmetic. In the
$24$-dimensional case, we used exact rational arithmetic. Either way,
it's not a big obstacle. The inequalities you need have a little
slack, which means you can bound everything in any number of different
ways, without needing to do anything too delicate.

\item[\sc Steve]

To elaborate: Our positivity check involves showing certain power
series $f(q)$ in a parameter $q$ are positive, where $q < 1$.  It is
not difficult to bound the coefficients and deduce this positivity for
$q < c$ (where $c$ is an effective constant), so the problem reduces
to showing positivity for $q$ in the interval $[c,1]$.  Numerically
one can plot this directly, of course.  From such a graph we see
$f(q)>b$ for some explicit constant $b$.  Write $f(q)=p(q)+t(q)$,
where $p(q)$ is a polynomial consisting of the first several terms in
the power series and $t(q)$ the tail, with the number of terms chosen
so that $t(q)$ is provably less than $b/2$ for $q$ in $[c,1]$.  We are
now reduced to showing $p(q)-b/2 > 0$ for $q \in [c,1]$, and such an
inequality can be rigorously established using Sturm's theorem.

\end{description}

\section*{Modular forms}

\begin{description}

\item[\sc Frank]

For a long time it has been known that the $\mathsf{E}_8$ lattice and
the Leech lattice have strong connections to modular forms, simply by
their theta series. However, one thing which puzzles us (we mainly
work in optimization) that you solve an optimization problem, an
infinite-dimensional linear program, with tools from analytic number
theory, especially using modular forms. How surprising was it to you
that using modular forms was one key to the proof?

\item[\sc Henry]

Modular forms are by far the most important class of special functions
related to lattices, so in that sense it's not so surprising that they
come up.  Over the years many people had suggested using them, but it
wasn't clear how.  For example, many years ago I had tried the Laplace
transform of a modular form, but without Maryna's $\sin^2$-factor.
Without that, it seemed impossible to get anything like the right
roots, and therefore the approach was completely useless.  What I find
beautiful about her proof is how ingeniously it puts everything
together (when I know from personal experience that thinking ``I'd
better use modular forms'' will not just lead you to this proof).

Before Maryna's proof, I could imagine two possibilities.  One was
that the right approach would be to solve the LP problem in general,
getting the exceptional dimensions just as special cases.  This
approach might not have involved modular forms at all.  The other was
that there would be particular special functions in those dimensions.
I always hoped for something special in $8$ and $24$ dimensions (e.g.,
based on the numerical experiments Steve and I worked on), but I was a
little worried that maybe the difficulty of writing these functions
down indicated that this might be the wrong approach (while solving
the problem in general seemed even harder).  It was great to see that
everything was as beautiful as we had always hoped.

\item[\sc Steve]

The use of Poisson summation in the Cohn-Elkies paper (and an earlier
technique of Siegel) had already brought methods of analytic number
theory into the subject.  It was clear relatively early on that
modular forms must be somehow involved in the final answer.  This is
both because of the appearance of special Bernoulli numbers (that
prominently arise in modular forms) as well as the overall structure
of the summation formulas.  For example, Henry and I derived a
relevant Fourier eigenfunction with simple zeros using Voronoi
summation formulas and derivatives of modular forms.

However, while these ingredients had been on the table for a long
time, we didn't know how to combine them until Maryna's paper
appeared.  At that point many of the various pieces of evidence we had
suddenly fit together.  Quasi-modular forms (which are derivatives of
modular forms) are very crucial to this story.

\end{description}

\section*{Collaboration}

\begin{description}

\item[\sc Frank]

Maryna's breakthrough paper which solved the sphere packing problem in
dimension $8$ was submitted on March 14, 2016 to the arXiv-preprint
server. Then it took only one week until you submitted the solution of
the sphere packing problem in dimension $24$ to the arXiv. Working at
five different, distant places, how did you collaborate?  What were
the difficulties when going from $8$ to $24$ dimensions?

\item[\sc Steve]

That was certainly an exciting and memorable week.  Once we had
assembled our team, things moved extremely quickly.  This is mainly
because Maryna's methods are so powerful, but it was also important
that certain pairs of us (Henry-myself, Abhinav-Henry, and
Danylo-Maryna) were established collaborators that had already worked
together well.

In addition to phone calls and email, we used \texttt{Skype} and
\texttt{Dropbox} to communicate our ideas.  I particularly like
drawing mathematics on a tablet PC and sharing the screen on
\texttt{Skype}---this allows the others to watch as if I'm writing on
a blackboard.  As soon as we saw Maryna's paper on Tuesday morning, we
tried to make concrete bridges with the numerical observations Henry
and I made in our Cohn-Miller paper.  After some reformulation of her
modular forms as quotients, by Wednesday it was then clear what
properties of the $q$-expansions would be needed to obtain the
$24$-dimensional functions.  We also set up computer programs to match
potential candidate functions with the numerical values that we could
compute separately.

On Thursday we had the right space of modular forms and a program
ready to search in them.  By matching the conjectured rationalities,
we found the even eigenfunction on Thursday night and the odd
eigenfunction on Friday morning.  We used \texttt{Mathematica} and
\texttt{PARI/GP} for this.  We also checked using graphs and
$q$-expansions that the necessary positivity conditions indeed hold,
but did not have a completely rigorous proof of this remaining point.

By Friday afternoon I was completely exhausted (and in any event do
not work on the Jewish Sabbath).  It's important to note that the
positivity analysis was a little different in $24$ dimensions than in
$8$ because of an extra pole that occurs.
\end{description}

\section*{Going further}

\begin{description}

\item[\sc Frank]

Now the sphere packing problem has been solved in $1$, $2$, $3$, $8$
and $24$ dimensions and, coming close to the end of the interview, it
is time to make speculations: Are there candidate dimensions where a
solution is in sight? Do you think that your method will be useful for
this (or for other problems)?

\item[\sc Henry]

It's hard to say for sure whether there might be further sharp cases
in higher dimensions, but it seems unlikely that they would have
remained undetected. (At the very least it's not plausible that they
could have the same widespread occurrences in mathematics as
$\mathsf{E}_8$ or the Leech lattice, since they would presumably have
been discovered in the process of classifying the finite simple
groups, for example.) Two dimensions, $n = 2$, remains open, although
of course a solution is known by elementary geometry; the LP bounds
behave rather differently in two dimensions compared with $8$ or $24$.

\item[\sc Steve]

Yes, dimension two seems very different because of the delicate
arithmetic nature of the root lengths.

\item[\sc Henry]
 
I'm sure Maryna's wonderful approach to constructing these functions
is just the tip of an iceberg, and I'm optimistic that humanity will
learn more about how LP bounds and related topics work.

One mystery I find particularly intriguing is what's so special about
$8$ and $24$ dimensions. Maryna's methods give a beautiful proof, but
I still don't really know a conceptual explanation as to what's
different in, say, $16$ dimensions (beyond just the fact that the
Barnes-Wall lattice isn't as nice as $\mathsf{E}_8$ or the Leech
lattice).  On a slightly different topic, I hope someone solves the
four-dimensional sphere packing problem, but it will require different
techniques. Unlike the cases where the LP bounds are sharp, these pair
correlation inequalities will not suffice by themselves. But the
$\mathsf{D}_4$ lattice is awfully beautiful, and the world deserves a
proof of optimality. The only question is how\dots

\item[\sc Frank]

Henry, Abhinav, Steve, and Maryna, thank you very much for this
interview.

\end{description}

\end{appendix}

\end{document}